\documentclass{article}
\usepackage{amsthm}
\usepackage{amsmath}
\usepackage{amsfonts}
\usepackage[letterpaper,body={14.5cm,23cm}, mag=1000]{geometry}
\usepackage{amssymb}
\usepackage{secdot}
\usepackage{setspace}
\usepackage{titlesec}
\titleformat{\section}[runin]{\bfseries\filcenter}{\thesection}{1em}{}

\renewcommand{\thesection}{\arabic{section}}
\title{\large \bf On the conjecture of non-inner automorphisms of  finite $p$-groups} 
\author{\small \bf Mandeep Singh$^\ast$ and Mahak Sharma$^{\ast\ast}$\\
\small \em $^\ast$Department of Mathematics, Arya College, Ludhiana - 141 001,
India\\
\small \em $^{\ast\ast}$Department of Mathematics, Goswami Ganesh Dutta Sanatan Dharma College, Chandigarh - 160019,
India\\
}
\date{}
\DeclareMathOperator{\Hom}{Hom}

\DeclareMathOperator{\Inn}{Inn}
\DeclareMathOperator{\Aut}{Aut}

\DeclareMathOperator{\Cent}{Aut_z}

\newtheorem{thm}{Theorem}[section]
\newtheorem{lm}[thm]{Lemma}
\newtheorem{cor}[thm]{Corollary}

\begin{document}
\maketitle

\begin{abstract}
\noindent Let $p$ be a prime number. A longstanding conjecture asserts that every finite non-abelian $p$-group has a non-inner automorphism of order $p$. In this paper, we prove that if $G$ is an odd order finite non-abelian monolithic $p$-group such that every maximal subgroup of $G$ is non-abelian and 
		$[Z(M), g] \leq Z(G)$ for every maximal subgroup $M$ of $G$ and $g \in G \setminus M$. Then $G$ has a non-inner automorphism of order $p$ leaving the Frattini subgroup $\Phi(G)$ elementwise fixed.
\end{abstract}

\vspace{2ex}

\vspace{2ex}


\section{Introduction}
Let $G$ be a finite non-abelian $p$-group, where $p$ is a prime number. In 1966, Gasch\"utz \cite{gas1966} proved that $G$ has a non-inner automorphism of $p$-power order. In 1973, Berkovich \cite{ber2008} proposed the following conjecture (see for eg. \cite[Problem 4.13]{mazkhu}). 
	\begin{center}
		{\bf Prove that every finite non-abelian p-group admits an automorphism of order $p$ which is not an inner one.}
	\end{center}
	Mathematicians are interested in finding the non-inner automorphism of order $p$ which fixes one of the subgroups $Z(G)$, $\Phi(G)$, $\Omega_1(Z(G))$ or $\Omega_1(\Phi(G))$ elementwise.
	In 1965, Liebeck \cite{lie1965} proved that if $G$ is of class 2, then $G$ has a non-inner automorphism of order $p$ if $p >2$, and order $2$ or $4$ if $p=2$ which fixes $\Phi(G)$ elementwise. This conjecture has been settled for the following classes of $p$-groups:
	$G$ is regular \cite{deasil2002, sch1980};
	the nilpotency class of $G$ is $2$ or $3$ \cite{abd2007, lie1965}; $C_G(Z(\Phi(G)))\neq \Phi(G)$ \cite{deasil2002}; $G/Z(G)$ is powerful \cite{abd2010}; $G$ is a $p$-group, where $p >2$ such that $(G, Z(G))$ is a Camina pair \cite{gho2013}; or $Z(G)$ is non-cyclic and $W(G)$ is non-abelian, where $W(G)/Z(G)=\Omega_1(Z_2(G)/Z(G))$ \cite{garsin2021}. For more details on this conjecture, the readers are advised to see (\cite{abdgho2014}, \cite{fouorf2014}, \cite{kom2024}, \cite{ruslegyad2016}) and references therein.
	
	It is observed that in most of the cases if the conjecture is false, then the center $Z(G)$ is cyclic. The main motivation behind this paper is the following natural question:\\
	
	{\bf Question 1}: Let $G$ be a finite non-abelian $p$-group with cyclic center $Z(G)$. Under what additional conditions on $G$ does the conjecture hold?\\
	
	Most recently, in 2024, Komma \cite{kom2024} answered this question by proving that if $G$ is an odd order $p$-group with cyclic center satisfying $C_G(G^p \gamma_3(G)) \cap Z_3(G)\leq Z(\Phi(G))$, then the conjecture is true.\\
	
	It follows from the main result of Attar \cite{att2009} that if $G$ is a finite non-abelian $p$-group such that $Z_2(G)/Z(G)$ is cyclic. Then $G$ has a non-inner central automorphism of order $p$ which fixes $\Phi(G)$ elementwise. Therefore the following natural question also arises.\\
	
	{\bf Question 2}: Let $G$ be a finite non-abelian $p$-group such that $Z_2(G)/Z(G)$ is non-cyclic abelian. Under what additional conditions on $G$ does the conjecture hold?\\
	
	In this paper, we give an affirmative answer to these questions in case of  monolithic $p$-groups. Infact, we prove that if $G$ is an odd order finite non-abelian monolithic $p$-group such that every maximal subgroup of $G$ is non-abelian and $[Z(M), g] \leq Z(G)$ for every maximal subgroup $M$ of $G$ and $g \in G \setminus M$, then $G$ has a non-inner automorphism of order $p$ leaving the Frattini subgroup $\Phi(G)$ elementwise fixed.

\section{Main results}
We first recall from \cite{hal1959} that a group $G$ is said to be monolithic if and only if there exists a non-trivial normal subgroup of $G$ which is contained in every non-trivial normal subgroup of $G$. Equivalently, a $p$-group $G$ is said to be monolithic if $Z(G) \simeq C_p$. For $x \in G$, $\lbrace [x,G] \rbrace$ denotes the set $\lbrace [x,g] \mid g \in G\rbrace$. It can be checked that for a finite non-abelian $p$-group $G$ if $[Z(M), g] \subseteq Z(G)$ for every non-abelian maximal subgroup $M$ of $G$ and $g \in G \setminus M$, then $[Z(M), g] \leq Z(G)$. All other unexplained notations are standard.\\

We will use the following well known lemmas.
	
\begin{lm}
Let $n \in N$, $x \in Z_2(G)$ and $y \in G$. Then 
\begin{enumerate}
\item[$(i)$] $(xy)^n=x^ny^n[y,x]^{\frac{n(n-1)}{2}}$.
\item[$(ii)$] $[x^n,y]=[x,y]^n=[x,y^n]$.
\end{enumerate}
\end{lm}

\begin{lm}
Let $G$ be a finite $p$-group, $M$ be a maximal subgroup of $G$ and $g \in G\setminus M$. Let $u \in Z(M)$ such that $(gu)^p=g^p$. Then the map $\alpha$ given by $g \rightarrow gu$ and $m \rightarrow m $, for all $m \in M$, can be extended to an automorphism of $G$ of order $p$ that fixes $M$ elementwise.
\end{lm}

\begin{thm}
Let $G$ be a finite non-abelian monolithic $p$-group $(p > 2)$ such that every maximal subgroup of $G$ is non-abelian and $[Z(M), g] \leq Z(G)$ for every maximal subgroup $M$ of $G$ and $g \in G \setminus M$. Then $G$ has a non-inner automorphism of order $p$ which fixes 
$\Phi(G)$ elementwise.
		
\end{thm}

	Notice that we cannot drop the condition that every maximal subgroup $M$ in the hypothesis is non-abelian. Because if $G$ has an abelian maximal subgroup $M$, then $C_G(Z(\Phi(G))) \neq \Phi(G)$. Suppose, if possible, $C_G(Z(\Phi(G))) = \Phi(G)$. Then $C_G(\Phi(G))=Z(\Phi(G))$. Therefore 
	$$M\leq C_G(M) \leq C_G(\Phi(G)) =Z(\Phi(G)) = \Phi(G).$$
	Thus $M= \Phi(G)$, which is a contradiction. It   follows from \cite{deasil2002} that $G$ has a non-inner automorphism of order $p$ which fixes $\Phi(G)$ elementwise.

\begin{proof}
Since  $G$ is a finite non-abelian monolithic $p$-group, $|Z(G)|=p$. Suppose if possible that $G$ has no non-inner automorphism of order $p$ leaving the Frattini subgroup $\Phi(G)$ elementwise fixed. It follows from \cite{att2009} that $Z_2(G)/Z(G)$ is not cyclic. Also by \cite{deasil2002}, $C_G(Z(\Phi(G)))=\Phi(G)$.
		
		We claim that $Z_2(G)$ is abelian. It follows from \cite[5.2.22]{rob1996} that $$\exp(Z_2(G)/Z(G)) \leq \exp(Z(G))=p.$$ Since $Z_2(G)/Z(G)=Z(G/Z(G))$ is abelian, $Z_2(G)/Z(G)$ is elementary abelian. Therefore $g^p \in Z(G)$ for all $g \in Z_2(G)$ which implies that $[x, g^p]=1$ for all $x \in G$. Now for $g \in Z_2(G)$ and $x \in G$, it follows from Lemma 2.1 that $[x^p, g]= [x, g]^p = [x, g^p]=1$. Also by Grun's Lemma $[Z_2(G), G^{\prime}]=1$. Thus $$[Z_2(G), \Phi(G)]= [Z_2(G), G^p][Z_2(G), G^{\prime}]=1.$$ Therefore $Z_2(G) \leq C_G(\Phi(G))= Z(\Phi(G))$. Thus $Z_2
		(G)$ is abelian. 
		
		Next we prove that $Z(M) \leq Z_2(G)$ for every maximal subgroup $M$ of $G$. Suppose, if possible that $Z(M)$ is not contained in $Z_2(G)$ for some maximal subgroup $M$ of $G$. Let $m \in Z(M) \setminus Z(G)$ such that $m \notin Z_2(G)$. Then there exists some $x \in G\setminus M$ such that $[m,x] \notin Z(G)$, which is a contradiction to the hypothesis.

		Next we prove that $Z_2(G)$ contains an element of order $p$ which does not belong to $Z(G)$. Since $d(Z(G))=1$, it follows from \cite[corollary 2.3]{abd2010} that $d(Z_2(G)/Z(G))=d(G)$. Observe that $Z_2(G)$ can not be cyclic. Because if $d(Z_2(G))=1$, then  $d(Z_2(G)/Z(G))=1$, and therefore
		\[
		d(G)= d(Z_2(G)/Z(G))=1,
		\]
		which is a contradiction. As $Z_2(G)$ is abelian, it implies that $d(\Omega_1(Z_2(G)))=d(Z_2(G)) \geq 2$.
		Since $|Z(G)|=p$, therefore $Z_2(G)$ contains an element of order $p$ which does not belong to $Z(G)$. \\
		Consider $u \in \Omega_1(Z_2(G)) \setminus Z(G)$. Define a map $\sigma: G\rightarrow Z(G)$ by $\sigma(x)= [x,u]$ for all $x \in G$. The map $\sigma $ is surjective with $\text{ker} \sigma = C_G(u)$. Therefore the subgroup $M= C_G(u)$ is a maximal subgroup of $G$. Let $g \in G \setminus M$, it follows from Lemma 2.1 that $$(gu)^p = g^p u^p[u, g]^\frac{p(p-1)}{2} = g^p.$$ By Lemma 2.2, the map $\alpha : G\rightarrow G$ defined by $\alpha(g) = gu$ and $\alpha(m)= m$ for all $m \in M$ can be extended to an automorphims of order $p$. By hypothesis $\alpha$ must be an inner automorphism. Let $\alpha = \beta_t$, the inner automorphim induced by $t$. Since $\alpha$ is the identity on $M$, $t \in C_G(M)$. If $t \notin M$, then $G=M \langle t \rangle$. It implies that $t \in Z(G) \leq \Phi(G) \leq M$, which is a contradiction. It follows that $t \in Z(M) \leq Z_2(G)$. Thus $u=g^{-1}\alpha(g)=[g, t] \in Z(G)$, which contradicts the choice of $u$. Hence $G$ has a non-inner automorphism of order $p$ which fixes $\Phi(G)$ elementwise.
\end{proof}

Attar \cite{att2009} proved that if $G$ is a finite non-abelian $p$-group such that $C_G(Z(\Phi(G))) = \Phi(G)$ and  $\frac{Z_2(G)\cap Z(\Phi(G))}{Z(G)}$ is not elementary abelian of rank $rs$, where $r=d(G)$ and $s=d(Z(G))$, then $G$ has a non-inner central automorphism of order $p$ which fixes $\Phi(G)$ elementwise. The following corollary, which is an immediate consequence of Theorem 2.3, confirms that the conjecture holds if $G$ is a finite non-abelian monolithic $p$-group of odd order such that $C_G(Z(\Phi(G))) = \Phi(G)$ and  $\frac{Z_2(G)\cap Z(\Phi(G))}{Z(G)}$ is elementary abelian of rank $r$, where $r=d(G)$.

\begin{cor}
Let $G$ be a finite non-abelian monolithic $p$-group $(p >2)$ such that $C_G(Z(\Phi(G))) = \Phi(G)$ and $[Z(M), g] \leq Z(G)$ for every maximal subgroup $M$ of $G$ and $g \in G \setminus M$. Then $G$ has a non-inner automorphism of order $p$ which fixes $\Phi(G)$ elementwise.
\end{cor}

\begin{proof}
Since $C_G(Z(\Phi(G))) = \Phi(G)$, every maximal subgroup of $G$ is non-abelian. It follows from the proof of Theorem 2.3 that $Z_2(G) \leq Z(\Phi(G))$. Thus $\frac{Z_2(G)\cap Z(\Phi(G))}{Z(G)}= \frac {Z_2(G)}{Z(G)}$ is elementary abelian and $d(\frac{Z_2(G)\cap Z(\Phi(G))}{Z(G)})=d(G)$. Rest of the proof is similar.
\end{proof}

	We conclude this paper with an example of a group that satisfies the hypothesis of  Theorem 2.3.

	Consider 
	$$ G= \langle f_1, f_2, f_3, f_4, f_5, f_6, f_7  \rangle,$$ with the relations: $
	[f_2,f_1]=f_3,
	[f_1,f_7]=[f_2,f_7]=[f_3,f_6]=[f_4,f_5]=1, 
	f_1^{3}=f_4, 
	f_2^{3}=f_3,
	f_3^{3}=f_5,
	f_4^{3}=f_6,
	f_5^{3}=f_7,
	f_6^{3}=f_7^{3}=1$. This group is the group number 194 in the GAP Library of groups of order $3^7$. The nilpotency class of $G$ is $4$. In this group $Z(G)=\langle f_7 \rangle$ and therefore $|Z(G)|=3$. Thus $G$ is a monolithic $3$-group. This group has four maximal subgroups $M_1, M_2, M_3$ and $M_4$, where
	
	$$M_1 = \langle f_1, f_3, f_4, f_5, f_6, f_7 \rangle,
	\ Z(M_1)= \langle f_6, f_7 \rangle,$$ 
	
	$$M_2 = \langle f_2, f_3, f_4, f_5, f_6, f_7 \rangle,
	\ Z(M_2)= \langle f_5, f_7 \rangle,$$ 
	
	$$M_3 = \langle f_1f_2^{2}, f_3, f_4, f_5, f_6, f_7 \rangle,
	\ Z(M_3)= \langle f_5^{2}f_6f_7, f_7^{2} \rangle,$$ 
	
	$$M_4 = \langle f_1f_2, f_3, f_4, f_5, f_6, f_7 \rangle,
\	Z(M_4)= \langle f_5f_6f_7^{2}, f_7 \rangle.$$

	Therefore $\Phi(G)= \langle  f_3, f_4, f_5, f_6, f_7 \rangle$. Observe that since $\Phi(G)$ is non-abelian, it implies that $M_1, M_2, M_3$ and $M_4$ are also non-abelian.

	One can check that $[Z(M), g] \leq Z(G)$ for all $M \in \lbrace M_1, M_2, M_3, M_4 \rbrace$ and $g \in G \setminus M$. This group has 4374 automorphisms. One of these automorphism is $\alpha$, where 
	$$
	\alpha(f_i)=f_i ~ \mbox{for all} ~ i \in \{1,3,4,5,6,7\} ~ \mbox{and} ~ \alpha(f_2)=f_2f_6.
	$$ 
Note that $\alpha$ is a non-inner automorphism of order 3 which fixes $\Phi(G)$ element-wise.

\end{document}